\documentclass[12pt]{amsart}

\usepackage{fullpage, graphicx, amsmath, amssymb}

\newtheorem{thm}{Theorem}[section]
\newtheorem{prop}[thm]{Proposition}

\newtheorem{lemma}[thm]{Lemma}

\newtheorem{definition}[thm]{Definition}
\theoremstyle{definition}

\title{A few $c_2$ invariants of circulant graphs}

\author{Karen Yeats}
\thanks{The author is supported by an NSERC discovery grant.}

\begin{document}

\begin{abstract}
The $c_2$ invariant is an arithmetic graph invariant introduced by Schnetz \cite{SFq} and further developed by Brown and Schnetz \cite{BrS} in order to better understand Feynman integrals.  

This document looks at the special case where the graph in question is a 4-regular circulant graph with one vertex removed; call such a graph a decompletion of a  circulant graph.  The $c_2$ invariant for the prime $2$ is computed in the case of the decompletion of circulant graphs  $C_n(1,3)$ and $C_{2k+2}(1,k)$.  For any prime $p$ and for the previous two families of circulant graphs along with the further families $C_n(1,4)$, $C_n(1,5)$, $C_n(1,6)$, $C_n(2,3)$, $C_n(2,4)$, $C_n(2,5)$, and $C_n(3,4)$, the same technique gives the $c_2$ invariant of the decompletions as the solution to a finite system of recurrence equations. 
\end{abstract}

\maketitle

\section{Introduction}

Let $\Gamma$ be a connected 4-regular graph and let $G= \Gamma \smallsetminus v$ for some $v \in V(\Gamma)$.  Assume $G$ is also connected.  The graph $\Gamma$ can be uniquely reconstructed from $G$.  Call $\Gamma$ the \textbf{completion} of $G$ and call $G$ a \textbf{decompletion} of $\Gamma$.  We can think of $G$ as a Feynman graph in $\phi^4$ theory with 4 external edges.  Use the notation $G = \widetilde{\Gamma}$ to indicate that $G$ is a decompletion of $\Gamma$.  In general this is bad notation since the decompletion is not unique, but for the graphs of primary interest here, namely certain circulant graphs, all decompletions are isomorphic and so $\widetilde{\Gamma}$ is well defined.

\begin{definition}
The \textbf{circulant graph} $C_n(i_1, i_2, \ldots, i_k)$ is the graph on $n$ vertices with an edge between vertices $i$ and $j$ if and only if $i-j = i_\ell \mod n$ or $j-i = i_\ell \mod n$ for some $1 \leq \ell \leq k$.
\end{definition}

Without loss of generality we can always assume that the $i_\ell$ are at most $n/2$.

The circulant graphs which are also 4-regular are of the form $C_n(i,j)$ with $i \neq n-j$ and $i,j \neq n/2$.  These will be our focus in what follows.  In the special case of $C_n(1,j)$ the edges between vertices differing by $1$ will be called the \textbf{circle edges} and the edges between vertices differing by $j$ will be called the \textbf{chord edges}.

\medskip

Returning to $G = \widetilde{\Gamma}$, assign a variable $a_e$ to each edge $e\in G$. Define the (dual) \textbf{Kirchhoff polynomial} or first Symanzik polynomial to be
\[
  \Psi_G = \sum_{T}\prod_{e \not \in T}a_e
\]
where the sum runs over all spanning trees of $G$.  Using this polynomial define the \textbf{Feynman period} of $G$ to be
\[
 \int_{a_i \geq 0} \frac{\Omega}{\Psi_G^2}
\]
where $\Omega = \sum_{i=1}^{|E(G)|}(-1)^i da_1 \cdots da_{i-1} da_{i+1} \cdots da_{|E(G)|}$.  This integral converges provided $\Gamma$ is internally 6-edge connected.  There has been a lot of interest in the Feynman period lately because it is a reasonable algebro-geometric, or even motivic, object \cite{bek, Brbig, Mmotives, Sphi4}, but it is also a key piece of the full Feynman integral.

In order to get a better grasp on these periods, Schnetz \cite{SFq} defined the following graph invariant based on counting points on the affine variety of $\Psi_G$ (see also \cite{BrS}).
\begin{definition}
Let $p$ be a prime, let $\mathbb{F}_p$ be the field with $p$ elements, and let $G$ have at least 3 vertices.  Let $[\Psi_G]_p$ be the number of points in the affine variety of $\Psi_G$ over $\mathbb{F}_p$.  Define the $c_2$-invariant of $G$ at $p$ to be
\[
  c_2^{(p)}(G) = \frac{[\Psi_G]_p}{p^2} \mod p.
\]
\end{definition}
The fact that this is well defined depends on $G$ having at least three vertices and is proved in \cite{BrS}.  The $c_2$ invariant has or is conjectured to have the same symmetries as the Feynman period \cite{BrS, Dc2} and hence is a useful tool to understand and predict properties of the period.

\medskip

From the perspective of Feynman periods and the $c_2$ invariant, circulant graphs include the simplest nontrivial class of graphs and the apparently most intractable graphs.  
The simplest nontrivial graphs are the \textbf{zigzag graphs} which are $\widetilde{C_n}(1,2)$.  The Feynman periods of zigzag graphs are completely understood \cite{Szigzag}.  They are proven to be multiples of odd Riemann zeta values with the coefficient an explicit expression in binomial coefficients.  Zigzag graphs have $c_2^{(p)}=-1$.  This can be seen in a variety of ways, for example by using the double triangle reduction \cite{BrY} or the vertex width \cite{Brbig} to see that the zigzags can have all edges denominator reduced.

In contrast some of the most difficult and most mysterious graphs, see the last entries at each size in the census of \cite{Sphi4}, are also circulant graphs.  Thus circulant graphs provide an interesting playground for a better combinatorial understanding of the $c_2$ invariant since they are very symmetric as graphs and yet they include both easy and very difficult graphs from the quantum field theory perspective.

The structure of this document is as follows.  After a section of preliminaries, the $c_2$ invariant of the zigzags will be recomputed.  Next $c_2^{(2)}(\widetilde{C_n}(1,3))$ will be computed for all $n$.  This calculation illustrates a general technique using recurrences in $n$.  This technique also applies in principle to $c_2^{(p)}(\widetilde{C_n}(1,3))$ for any fixed prime $p$ as well as to $c_2^{(p)}(\widetilde{C_n}(1,k))$ for $k \leq 6$ and to $c_2^{(p)}(\widetilde{C_n}(j,k))$ for $(j,k)\in \{(2,3), (2,4), (2,5), (3,5)\}$ for any fixed $p$.  However, the calculations quickly become impossibly large.  A recurrence is given explicitly for $c_2^{(2)}(\widetilde{C_n}(2,3))$ to illustrate this.  The method fails for other $(j,k)$ which are constant in $n$.  Next $c_2^{(2)}(\widetilde{C_{2k+2}}(1,k))$ is computed explicitly and by similar arguments $c_2^{(p)}(\widetilde{C_{2k+2}}(1,k))$ is computable in principle for any fixed $p$.

\section{Preliminaries}

We will need some polynomials to do the $c_2$ calculations in the next sections.  By the matrix tree theorem $\Psi_G$ can be represented as a determinant in the following way.  Choose an arbitrary orientation for the edges of $G$ and let $E$ be the signed incidence matrix of $G$ (with rows indexing the vertices and columns indexing the edges) with one row removed.  Let $\Lambda$ be the diagonal matrix with the edge variables of $G$ on the diagonal.  Let
\[
M = \begin{bmatrix} \Lambda & E^T \\
  -E & 0 \end{bmatrix}.
\]
Then 
\[
\Psi_G = \det M.
\]
This can be seen directly by expanding out the determinant, see \cite{Brbig} Proposition 21, or by using the Schur complement and the Cauchy-Binet formula, see \cite{VY}.  In either case it comes down to the fact that the square full rank minors of $E$ are $\pm 1$ for columns corresponding to the edges of a spanning tree of $G$ and $0$ otherwise.  This fact is the essence of the matrix tree theorem.

If $I$ and $J$ are sets of indices then $M(I,J)$ is the matrix $M$ with rows indexed by elements of $I$ removed and columns indexed by elements of $J$ removed. Using this, we can define the \textbf{Dodgson polynomials} following Brown \cite{Brbig}.
\begin{definition}
  Let $I$, $J$ and $K$ be subsets of $\{1,2,\ldots, |E(G)|\}$ with $|I|=|J|$.  Define
\[
\Psi^{I,J}_{G,K} = \det M(I,J) |_{\substack{a_e = 0 \\ e\in K}}. 
\]
\end{definition}
When the graph is clear we will leave out the $G$ subscript.  When $K = \varnothing$ we will also leave it out.  Note that if $e \in I\cap J$ then both the row and column corresponding to $e$ are removed so the calculation is just as if $e$ were not even there.  That is
\[
\Psi^{Ie,Je}_{G,K} = \Psi^{I,J}_{G\smallsetminus e, K}.
\]
On the other hand if $e \in K$ but $e \not\in I \cup J$ then edge $e$ is unaffected by the row and column deletions but is set to zero.  This is saying that we are only taking monomials where $e$ does not appear, equivalently monomials where $e$ is not cut in the spanning structure.  That is
\[
\Psi^{I,J}_{G,eK} = \Psi^{I,J}_{G/e, K}
\]
when $e\not\in I\cup J$.  See \cite{Brbig} for full details and further identities of Dodgson polynomials.

In view of the all minors matrix tree theorem \cite{Chai}, Dodgson polynomials can also be expressed in terms of spanning forests.  The following spanning forest polynomials are convenient for this purpose.
\begin{definition}
Let $P$ be a set partition of a subset of the vertices of $G$.  Define 
\[
\Phi^P_G = \sum_{F} \prod_{e\not\in F}a_e
\]
where the sum runs over spanning forests $F$ of $G$ with a bijection between the trees of $F$ and the parts of $P$ where each vertex in a part lies in its corresponding tree. 
\end{definition}
Note that trees consisting of isolated vertices are permitted.  Also, it is important to keep in mind that a variable appearing \emph{in} a monomial corresponds to the edge being \emph{not in} the spanning forest.

\begin{figure}
\includegraphics{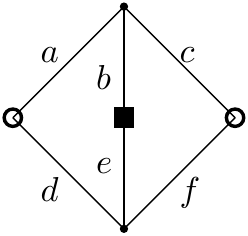}
\caption{A spanning forest example.}\label{spanning forest eg}
\end{figure}

In figures the partition will be illustrated by using a different large vertex shape for each part.  For example consider the graph in Figure~\ref{spanning forest eg} with the illustrated partition $P$.  Then
\[
\Phi^P = b(de+df+ef) + e(ab+ac+bc).
\]

The relationship between Dodgson polynomials and spanning forest polynomials
is given in \cite{BrY}.
\begin{prop}[Proposition 12 from \cite{BrY}]
Let $I,J,K$ be sets of edge indices of $G$ with $|I| = |J|$, then 
\[
\Psi^{I,J}_{G,K} = \sum_{P}\pm \Phi^{P}_{G\smallsetminus (I \cup J \cup K)}
\]
where the sum runs over all set partitions $P$ of the end points of edges of $(I \cup J \cup K) \smallsetminus (I \cap J)$
with the property that all the forests corresponding to 
$P$
become trees in both $G\smallsetminus I / (J\cup K)$ and $G \smallsetminus J / (I \cup K)$.
\end{prop}
This proposition is typical for how we will use spanning forest polynomials.  The graph subscript will often be dropped in $\Phi^P$ to keep the notation lighter.  In this case the underlying graph for $\Phi^P$ will be assumed to be $G$ with all the edges which have been worked with explicitly in the argument so far removed.

Dodgson polynomials are useful for computing the $c_2$ invariant.
\begin{definition}
Let $i,j,k,\ell,m$ be distinct edge indices of $G$.
The \textbf{5-invariant} of $G$ depending on $i,j,k,\ell,m$ is
\[
{}^5\Psi(i,j,k,l,m) = \pm(\Psi^{ij,k\ell}_m\Psi^{ikm,j\ell m} - \Psi^{ik,j\ell}_m\Psi^{ijm,k\ell m}).
\]
which is independent (up to sign) of the order of $i,j,k,\ell, m$ by Lemma 87 in \cite{Brbig}. 
\end{definition}

\begin{prop}\label{get started prop}
  Suppose $G$ has $2 + |E(G)| \leq 2|V(G)|$. Let $i,j,k,\ell, m$ be distinct edge indices of $G$.  Let $p$ be prime.
\begin{enumerate}
\item $c_2^{(p)}(G) = -[\Psi^{i,j}_k\Psi^{ik, jk}]_p \mod p$,
\item $c_2^{(p)}(G) = [\Psi^{ij,k\ell}\Psi^{ik,j\ell}]_p \mod p$, and
\item $c_2^{(p)}(G) = -[{}^5\Psi(i,j,k,\ell,m)]_p \mod p$
\end{enumerate}
where $[\cdot]_p$ again denotes counting points on the affine variety over $\mathbb{F}_p$.
\end{prop}

\begin{proof}
This follows from Lemma 24 of \cite{BrS} and the statement and proof of Corollary 28 of \cite{BrS}.
\end{proof}

If things are nice we can continue systematically.  Given $G$ with at least 5 edges a \textbf{denominator reduction} is a sequence of polynomials $D^5, D^6, \ldots, D^k$ depending on an order of the edges of $G$ defined by 
\begin{itemize}
  \item $D^5 = {}^5\Psi(1,2,3,4,5)$.
  \item If $D^j$ can be factored as $D^j=(Aa_{j+1}+B)(Ca_{j+1}+D)$ where $A$, $B$, $C$, and $D$ are polynomials not involving $a_{j+1}$ then set $D^{j+1} = \pm(AD-BC)$.  This step is called \textbf{reducing} edge $j+1$.
  \item If $D^{j+1} =0$ or $D^j$ cannot be factored then denominator reduction ends.
\end{itemize}
Note that the $D^j$ are defined up to sign.  Different orders on the edges will give different sequences of polynomials and the sequences may be of different lengths.  Note also that performing one reduction step on $\Psi^{ij,k\ell}\Psi^{ik,j\ell}$ gives the 5-invariant, so $\Psi^{ij,k\ell}\Psi^{ik,j\ell}$ can be thought of as $D^4$, though unlike the other $D^j$ it depends in more than sign on the order of $i,j,k,\ell$.

The $D^j$ polynomials for $j\geq 5$ are the denominators when integrating the Feynman period one edge at a time \cite{Brbig} and they also compute the $c_2$ invariant in the sense that if $2 + |E(G)| \leq 2|V(G)|$ then
\[
 c_2^{(p)}(G) = (-1)^n[D^n]_p \mod p
\]
whenever $D^n$ is defined, see \cite{BrS} Theorem 29. 

 We will not be making much use of denominator reduction in what follows, but it gives an easy way to see that if a variable appears only linearly in some $D^j$ or factors out of $D^j$ then reducing that variable is the same as taking its linear coefficient.  This is because in those cases either $B$ or $C$ is zero in the factorization above and so the next denominator is $\pm AD$.  

This observation has a nice interpretation at the level of spanning forest polynomials.  Suppose we are looking at a product of sums of spanning forest polynomials.  Suppose in one factor all the polynomials require a given edge $e$ to be in the forests.  This means that $e$ does not appear in this factor which puts us in the $C=0$ case.  So we only care about the $A$ part of the other factor; that is, the part where $e$ must appear in the polynomial, hence must be deleted from the graph.  Similarly, if in one factor all the polynomials require $e$ to be deleted then we are in the $B=0$ case and so we only care about the part of the other factor with $e$ forced to be in the forests, or equivalently with $e$ contracted.

A tool that will be used more heavily is the following result\footnote{Thanks to Francis Brown for pointing this result out to me.}
\begin{lemma}\label{coeff count lemma}
Let $F$ be a polynomial of degree $N$ in $N$ variables with integer coefficients.  Then the coefficient of $x_1^{p-1}\ldots x_N^{p-1}$ in $F^{p-1}$ is $[F]_p$ modulo $p$.
\end{lemma}

\begin{proof}
This is a consequence of one of the standard proofs of the Chevalley-Warning theorem.  See for instance section 2 of \cite{Ax}.  The proof runs as follows.  Each element of $\mathbb{F}_p$ is a $p-1$ root of $1$ or is $0$.  Therefore
\[
  [F]_p = \sum_{x\in \mathbb{F}_p^N}(1-F^{p-1}(\mathbf{x})) = -\sum_{x\in \mathbb{F}_p^N} F^{p-1}(\mathbf{x})
\] 
in $\mathbb{F}_p$.
Say $F$ has degree $d$.  Take any monomial $\mathbf{x}^\mathbf{u}$ of $F^{p-1}$; $\mathbf{x}^\mathbf{u}$ has degree at most $d(p-1)$.
\[
\sum_{x\in  \mathbb{F}_p^N} \mathbf{x}^\mathbf{u} = \prod_{i=1}^N \sum_{x_i\in \mathbb{F}_p}x_i^{u_i} = \prod_{i=1}^N Y(u_i)
\]
where $Y(u_i) = \begin{cases} -1 & \text{if $u_i$ is a positive multiple of $p-1$} \\
  0 & \text{otherwise.} \end{cases}$

The proof of the Chevalley-Warning theorem is then completed by the observation that if $d < N$ then for any monomial at least one of the $Y(u_i)$ is $0$ and hence so is $[F]_p$.  For the present lemma the proof is completed by the observation that if $d=N$ then the only monomial with a nonzero contribution is $x_1^{p-1}\ldots x_N^{p-1}$.
\end{proof}

\section{$\widetilde{C_n}(1,2)$ -- zigzags}\label{zigzag sec}

Now we are ready to begin calculating $c_2$ invariants of decompleted circulant graphs.  The simplest case is the decompletion of $C_n(1,2)$.  These graphs are known as zigzag graphs \cite{bkphi4}.  Their Feynman periods \cite{Szigzag} and $c_2$ invariants \cite{BrS} are known.  As a warm-up lets recalculate $c_2^{(p)}(\widetilde{C_n}(1,2))$ using techniques which will be helpful for other circulant graphs.

\begin{figure}
\includegraphics{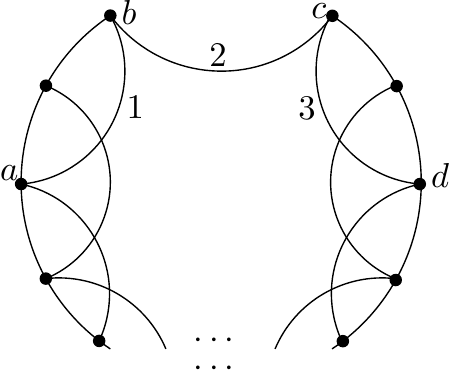}
\caption{$\widetilde{C_n}(1,2)$}\label{C_n(1,2)}
\end{figure}

\begin{prop}
 $c_2^{(p)}(\widetilde{C_n}(1,2)) = -1$ for all primes $p$ and all $n\geq 5$.
\end{prop}

\begin{proof}
  Label $\widetilde{C_n}(1,2)$ as in Figure~\ref{C_n(1,2)}.  In this figure as in the others in this paper, the vertex of $C_n(1,2)$ which was removed to form $\widetilde{C_n}(1,2)$ would be at the top.  Calculate 
\[
\Psi^{1,3}_2 = \pm\Phi^{\{a,d\}, \{b\}, \{c\}} = \pm\raisebox{-1cm}{\includegraphics{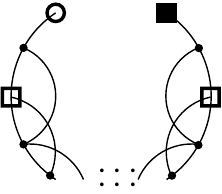}}
\] and 
\[
\Psi^{12,32} = \pm\left(\Phi^{\{a,d\},\{b,c\}} - \Phi^{\{a,c\},\{b,d\}}\right)
= \pm \left(\raisebox{-1cm}{\includegraphics{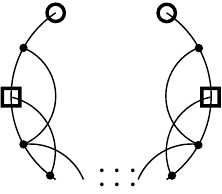}} - \raisebox{-1cm}{\includegraphics{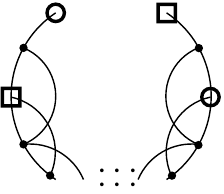}}\right).
\]
where the relative sign is explained in Corollary 17 of \cite{BrY}.
Consider the two terms in $\Psi^{12,32}$.  There must be paths connecting the vertices as indicated.  Suppose one of these paths used a circle edge which was not at either end.  Removing this circle edge along with its two incident vertices disconnects the graph and separates the two ends of the other path which is impossible.  Using all the other edges, namely all the chord edges along with the two end circle edges gives a spanning forest in $\Psi^{12,32}$ and so this must be the only spanning forest in $\Psi^{12,32}$.  Which of the two spanning forest polynomials this spanning forest belongs to depends on the parity.

Thus $(\Psi^{12,32})^{p-1}$ contributes $x_i^{p-1}$ to the coefficient of $(x_1\ldots x_{|E(G)|})^{p-1}$ in $(\Psi^{1,3}_2\Psi^{12,32})^{p-1}$ for each non-end circle edge $i$.  This leaves the chord edges and end edges each to the power $p-1$ to come from $(\Psi^{1,3}_2)^{p-1}$.  By linearity this could only occur by each $\Psi^{1,3}_2$ contributing each of these variables exactly once.  This corresponds to the spanning forest consisting of the non-end circle edges which is in $\Psi^{1,3}_2$. 

Therefore the coefficient of $(x_1\ldots x_{|E(G)|})^{p-1}$ in $(\Psi^{1,3}_2\Psi^{12,32})^{p-1}$ is $1$ and so by Proposition~\ref{get started prop} and Lemma~\ref{coeff count lemma} $c_2^{(p)}(C_n(1,2)) = -1$ for all primes $p$ and all $n\geq 5$.
\end{proof}

\section{$\widetilde{C_n}(1,3)$}\label{C13p2}

\begin{figure}
\includegraphics{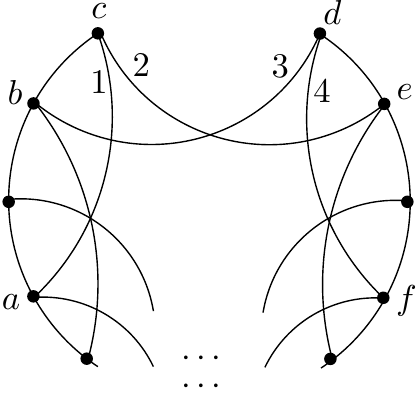}
\caption{$\widetilde{C_n}(1,3)$}\label{C_n(1,3)}
\end{figure}

Next consider the decompleted circulant graphs $\widetilde{C_n}(1,3)$ for $n\geq 7$, labelled as in Figure \ref{C_n(1,3)}.  Here we will work with $\Psi^{12,34}\Psi^{13,24}$.  We have
\[
\Psi^{12,34} = \pm\Phi^{\{b,c\}, \{d,e\}, \{a,f\}} = \pm\raisebox{-1cm}{\includegraphics{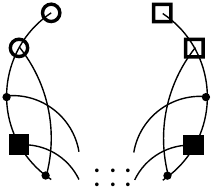}}.
\]
Any spanning forest in $\Psi^{12,34}$ must include the edge joining vertices $b$ and $c$ and the edge joining vertices $d$ and $e$ and so by denominator reduction or by Lemma \ref{coeff count lemma} we only need to consider the terms in $\Psi^{13,24}$ with those two edges deleted.  Calculating then,
\[
\Psi^{13,24} = \pm\Phi^{\{a,b,e,f\}, \{c\}, \{d\}} + \text{irrelevant terms}
= \pm\raisebox{-1cm}{\includegraphics{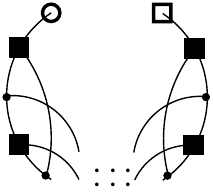}} + \text{irrelevant terms}.
\]
Reducing those two edges then, it suffices to consider
\[
\raisebox{-0.8cm}{\includegraphics{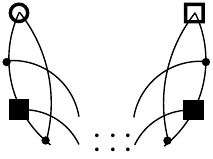}}\raisebox{-0.8cm}{\includegraphics{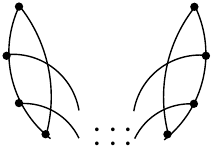}} = \Phi_H^{\{a,f\}, \{b\}, \{e\}}\Psi_H
\]
where $H$ is $G$ with edges $1,2,3,4$ and edges $(b,c)$ and $(d,e)$ deleted and isolated vertices removed.  The above all holds for any value of $p$, but now let us restrict to $p=2$.

\begin{prop}
   $c_2^{(2)}(\widetilde{C_n}(1,3)) \equiv n \mod 2$
 for $n\geq 7$.
\end{prop}

\begin{figure}
\includegraphics{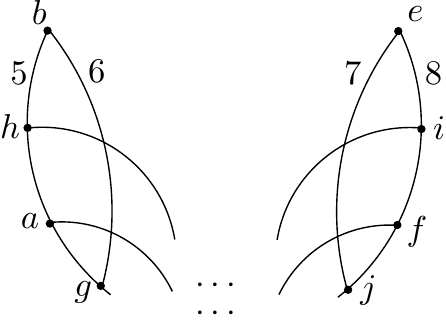}
\caption{$H$}\label{H}
\end{figure}

\begin{proof}
  Use notation from the preceding discussion.  Label $H$ as in Figure~\ref{H}; where useful, write $H_n$ to explicitly indicate the $n$ dependence of $H$.  Let $a_n = [\Phi_{H_n}^{\{a,f\}, \{b\}, \{e\}}\Psi_{H_n}]_p$.  In view of Lemma \ref{coeff count lemma} we need to assign each edge of $H$ to one of  $\Phi_H^{\{a,f\}, \{b\}, \{e\}}$ or $\Psi_H$; say an edge is assigned if the edge appears in corresponding the spanning forest or spanning tree, that is if the variable does not appear in that monomial.  If the assignment of edges is not symmetric under the left-right symmetry of $H$ then the flipped assignment is distinct and valid and so modulo $2$ these assignments are irrelevant.  

Consider edges $5$ and $6$.  To avoid disconnecting vertex $b$ at least one of these edges must be assigned to $\Psi_H$.  This gives three possibilities.  
\begin{itemize}
  \item \textbf{Edges $5$ and $6$ both assigned to $\Psi_H$:} By symmetry edges $7$ and $8$ are also assigned to $\Psi_H$.  

In this case vertices $b$ and $e$ are disconnected in $\Phi_H^{\{a,f\}, \{b\}, \{e\}}$, so the trees for parts $\{b\}$ and $\{e\}$ are singleton vertices and the rest of the graph must be spanned by the remaining part.  Hence, the remaining polynomial is $\Psi_{H_{n-2}}$.  

On the other hand, the part of $\Psi_H$ with edges $5$ and $6$ is the same as contracting edges $5$ and $6$ and similarly on the other side.  This is also the same as removing edges $5$ and $6$ and the isolated vertex while identifying vertices $h$ and $g$ and similarly on the other side.  This last way of looking at it gives all spanning forests which become trees when these two vertex identifications are made, which can again be interpreted on $H_{n-2}$.  Specifically we get
\[
\pm \Phi_{H_{n-2}}^{\{b,e\}, \{a\}, \{f\}} \pm \Phi_{H_{n-2}}^{\{b,f\}, \{a\}, \{e\}} \pm \Phi_{H_{n-2}}^{\{a,e\}, \{b\}, \{f\}} \pm \Phi_{H_{n-2}}^{\{a,f\}, \{b\}, \{e\}}.
\]
Note that the vertex labels are for $H_{n-2}$ labelled as in Figure~\ref{H} not for $H_{n-2}$ as a subgraph of $H_n$.
The middle two terms are equal by symmetry so are irrelevant modulo $2$, as are the signs.  All together from this case we are left with
\[
\left[\Psi_{H_{n-2}}\left(\Phi_{H_{n-2}}^{\{b,e\}, \{a\}, \{f\}} + \Phi_{H_{n-2}}^{\{a,f\}, \{b\}, \{e\}}\right)\right]_p = [\Psi_{H_{n-2}}\Phi_{H_{n-2}}^{\{b,e\}, \{a\}, \{f\}}]_p + a_{n-2}.
\]
\item \textbf{Only edge $6$ is assigned to $\Psi_H$:} By symmetry edges $7$ is also assigned to $\Psi_H$ but not edge $8$.  The rest of $\Psi_H$ remains a single spanning tree with no restrictions and so is $\Psi_{H_{n-2}}$.  From $\Phi_H^{\{a,f\}, \{b\}, \{e\}}$ we keep edges $5$ and $8$, or equivalently contract them, putting $h$ and $i$ in their own parts and leaving $\{a,f\}$ in a part together.  Viewed on $H_{n-2}$ (labelled as in Figure~\ref{H}) this is $\Phi_{H_{n-2}}^{\{b\}, \{h,i\}, \{e\}}$.  All together this is 
\[
[\Phi_{H_{n-2}}^{\{b\}, \{h,i\}, \{e\}}\Psi_{H_{n-2}}]_p.
\]

Now we need to take another step from $H_{n-2}$ to $H_{n-4}$.  To keep the parts separate edges $5$ and $8$ must not be assigned to $\Phi_{H_{n-2}}^{\{b\}, \{h,i\}, \{e\}}$, so there are two ways to proceed.  If edges $6$ and symmetrically $7$ are also assigned to $\Psi_{H_{n-2}}$ then similarly to the first case of the overall calculation we obtain $\Psi_{H_{n-4}}$ from $\Phi_{H_{n-2}}^{\{b\}, \{h,i\}, \{e\}}$ and $\Phi_{H_{n-4}}^{\{b,e\}, \{a\}, \{f\}} + \Phi_{H_{n-4}}^{\{a,f\}, \{b\}, \{e\}}$ from $\Psi_{H_{n-2}}$.  If edges $6$ and symmetrically $7$ are assigned to $\Phi_{H_{n-2}}^{\{b\}, \{h,i\}, \{e\}}$ then simply contracting the assigned edges and deleting the others we obtain $\Phi_{H_{n-4}}^{\{b,e\}, \{a\}, \{f\}}\Psi_{H_{n-4}}$.  Adding both subcases together the terms involving $\Phi_{H_{n-4}}^{\{b,e\}, \{a\}, \{f\}}$ cancel modulo $2$ and for this overall case we are left with
\[
  a_{n-4}.
\]
\item \textbf{Only edge $5$ is assigned to $\Psi_H$:} By symmetry edges $8$ is also assigned to $\Psi_H$ but not edge $7$.  Again, the rest of $\Psi_H$ remains a single spanning tree with no restrictions and so is $\Psi_{H_{n-2}}$, while $\Phi_H^{\{a,f\}, \{b\}, \{e\}}$ becomes $\Phi_{H_{n-2}}^{\{h,i\}, \{a\}, \{f\}}$.

Taking the next step to $H_{n-4}$, note that each factor needs one of edges $5$ or $6$ to avoid disconnecting $b$, and it doesn't matter which, since cutting one and contracting the other gives the same graph either way and does not involve any of the parts of the partition.  Thus this case has an even contribution.
\end{itemize}
Finally, let us make the same sort of argument on $[\Psi_{H_{n-2}}\Phi_{H_{n-2}}^{\{b,e\}, \{a\}, \{f\}}]_p$.  There are only two cases due to the need to connect vertices $b$ and $e$ in the second factor.  Assigning edges $6$ and $7$ to $\Phi_{H_{n-2}}^{\{b,e\}, \{a\}, \{f\}}$ results in a situation where neither factor specifies a part for $b$ or $e$ and so gives an even result as in the third main case.  Assigning edge $5$ and $8$ to $\Phi_{H_{n-2}}^{\{b,e\}, \{a\}, \{f\}}$ gives
\[
[\Phi_{H_{n-4}}^{\{h\}, \{i\}, \{b,e\}}\Psi_{H_{n-4}}]_p
\]
which can only be reduced to $H_{n-6}$ in one way giving
\[
[\Phi_{H_{n-6}}^{\{b\}, \{a,f\}, \{e\}}\Psi_{H_{n-6}}]_p = a_{n-6}.
\]

Putting everything together we get
\[
c_2^{(2)}(\widetilde{C_n}(1,3)) = a_n = a_{n-2} + a_{n-4} + a_{n-6}
\]
for $n\geq 13$.
By direct computation we can check that $c_2^{(2)}(\widetilde{C_7}(1,3)) = c_2^{(2)}(\widetilde{C_9}(1,3)) = c_2^{(2)}(\widetilde{C_{11}}(1,3)) = 1$ and  $c_2^{(2)}(\widetilde{C_8}(1,3)) = c_2^{(2)}(\widetilde{C_{10}}(1,3)) = c_2^{(2)}(\widetilde{C_{12}}(1,3))=0$.  Since we are working in $\mathbb{F}_2$ this gives the result.
\end{proof}

\begin{prop}\label{finite no 1}
  Fix any prime $p$.  It is a finite calculation to determine
   $c_2^{(p)}(\widetilde{C_n}(1,3))$ for all $n$.
\end{prop}

\begin{proof}
  Let $H$ be as before and write $H_n$ when the $n$ dependence is important.
  By the calculations at the beginning of this section along with Lemma \ref{coeff count lemma}, we need to compute the coefficient of $(x_1\cdots x_{|E(H_n)|})^{p-1}$ in  $(\Phi_{H_n}^{\{a,f\}, \{b\}, \{e\}}\Psi_{H_n})^{p-1}$ for all $n$.

  As in the case $p=2$ we can approach this problem by considering all possible assignments of $p-1$ copies edges $5$, $6$, $7$, and $8$ to the factors of $(\Phi_{H_n}^{\{a,f\}, \{b\}, \{e\}}\Psi_{H_n})^{p-1}$.  We can no longer restrict to symmetrical assignments, but none the less there are finitely many possibilities for any fixed $p$.

  For each factor, the result of any such assignment is to take a spanning forest polynomial on $H_n$ where a subset of the vertices $a,h,b,e,i,f$ participate in the parts and delete or contract each of $5$, $6$, $7$, $8$.  The result of this is a sum of spanning forest polynomials on $H_{n-2}$ where only vertices $a,h,b,e,i,f$ (labelled as in Figure~\ref{H}, not as a subgraph of $H_{n}$) can participate in the parts.  There are only finitely many set partitions of subsets of $a,h,b,e,i,f$.

  Let $\mathcal{P}$ be the set of all set partitions of subsets of $\{a,h,b,e,i,g\}$.  A multiset of $2(p-1)$ elements of $\mathcal{P}$ can be viewed as a product of spanning forest polynomials on $H$ with the given partitions.  Consider the digraph $D$ with vertices indexed by $2(p-2)$-tuples of elements of $\mathcal{P}$ and with a directed edge between two vertices if interpreting the first as a product $\Phi$ of spanning forest polynomial on $H_n$ there is a choice of assignment of $p-1$ copies edges $5$, $6$, $7$, and $8$ to the factors of $\Phi$ so that processing these edges gives the second vertex, interpreted as a product of spanning forest polynomials on $H_{n-2}$, as a summand.  The directed edge can then be labelled with the coefficient of this summand.  This directed graph has finitely many vertices and finitely many edges.

  $D$, then, determines a system of linear homogeneous recurrences with constant coefficients in the following way.  Associate a sequence of variables $a^{(v)}_i$ to each vertex $v$ of $D$ and let $w_e$ be the label of edge $e$ of $D$.  Vertex $v$ of $D$ yields the equation $a^{(v)}_i = \sum_{v'} w_{(v,v')}a^{(v')}_{i-2}$ where the sum runs over vertices in the out-neighbourhood of $v$.  The collection of all these equations for all vertices of $D$ gives a system of linear homogeneous recurrences with coefficients in the field $\mathbb{F}_p$.  Such systems are always solvable by standard techniques\footnote{For example, write the system in matrix form, $\mathbf{a}_n = A\mathbf{a}_{n-2}$, put $A$ in Jordan form over an appropriate extension of $\mathbb{F}_p$, then the $n$th power of a Jordan block has a closed form expression and so writing the $n$th power of $A$ gives the solution.}, and the solution associated to the vertex $(\Phi_{H}^{\{a,f\}, \{b\}, \{e\}}\Psi_{H})^{p-1}$ gives $c_2^{(p)}(\widetilde{C_n}(1,3))$.


\end{proof}

Note that nothing given here guarantees that this calculation will be doable in practice.  The digraph for $p=3$ will have around 100 vertices and the number of vertices is exponential in $p$.  So while $p=3$ should be doable with computer help, $p=5$ may not be doable unless there is further structure we can exploit.  Not all these vertices may be necessary -- it suffices to use those which can be reached following a directed path from $(\Phi_{H}^{\{a,f\}, \{b\}, \{e\}}\Psi_{H})^{p-1}$.  If the number of necessary vertices grows slowly, then higher values of $p$ may be calculable without further insight, otherwise new insights will be needed.

\section{$\widetilde{C_n}(2,3)$}

\begin{figure}
\includegraphics{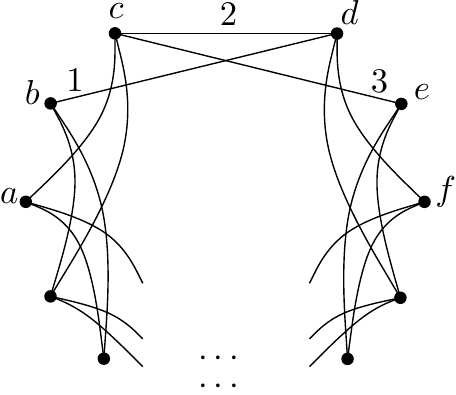}
\caption{$\widetilde{C_n}(2,3)$}\label{C_n(2,3)}
\end{figure}

This section considers $\widetilde{C_n}(2,3)$, labelled as in Figure \ref{C_n(2,3)}. 

\begin{prop}
Fix any prime $p$.  It is a finite calculation to determine
   $c_2^{(p)}(\widetilde{C_n}(2,3))$ for all $n$.
\end{prop}

\begin{proof}
The first step is to compute the starting polynomials, in this case $\Psi^{1,3}_2$ and $\Psi^{12,32}$, see Lemma~\ref{get started prop}.

\[
\Psi^{1,3}_2 = \pm\Phi_{H_n}^{\{b,e\}, \{c\}, \{d\}}
\]
and
\[
\Psi^{12,32} = \pm\left(\Phi_{H_n}^{\{c,d\},\{b,e\}} - \Phi_{H_n}^{\{c,b\}, \{d,e\}}\right).
\]
The relative signs are well defined, see \cite{BrY} Definition 15, Proposition 16, and Corollary 17 for details.

Let $H_n$ be $\widetilde{C_n}(2,3)$ with edges $1$, $2$, and $3$ removed.  We can proceed recursively as in the previous section.  Assigning the two remaining edges out of $c$ and $d$ among the factors of $ \left(\Psi^{1,3}_2\Psi^{12,32})^{p-1}\right)$ leaves $H_{n-2}$ with only the top three vertices on each side involved in any relevant spanning forest polynomial.  Therefore, the procedure described in Proposition~\ref{finite no 1} works in this case and hence for any fixed value of $p$, the calculation is finite.
\end{proof}

\allowdisplaybreaks

Now consider $p=2$ specifically.  We can compute the recurrence in this case.   Note that, as before, since $p=2$, we can further restrict to symmetric assignments of edges.  There are 22 linear combinations of spanning forest polynomials which play a role in the calculation.  Label them as follows
\begin{align*}
  a_n^{(1)} & = \Phi_{H_n}^{\{b,e\}, \{c\}, \{d\}}\\
  a_n^{(2)} & = \Phi_{H_n}^{\{c,d\}, \{b\}, \{e\}}\\
  a_n^{(3)} & = \Phi_{H_n}^{\{a,f\},\{b\},\{e\}}\\
  a_n^{(4)} & = \Phi_{H_n}^{\{b,e\}, \{a\}, \{f\}}\\
  a_n^{(5)} & = \Phi_{H_n}^{\{c,d\}, \{a\}, \{f\}}\\
  a_n^{(6)} & = \Phi_{H_n}^{\{a,f\}, \{c\}, \{d\}}\\
  b_n^{(1)} & = \Phi_{H_n}^{\{b,e\}, \{c,d\}} - \Phi_{H_n}^{\{b,c\}, \{d,e\}} \\
  b_n^{(2)} & = \Phi_{H_n}^{\{a,f\}, \{c,d\}} - \Phi_{H_n}^{\{a,c\}, \{d,f\}}\\
  b_n^{(3)} & = \Phi_{H_n}^{\{a,f\}, \{b,e\}} - \Phi_{H_n}^{\{a,b\}, \{e,f\}}\\
  d_n^{(1)} & = \Phi_{H_n}^{\{b,e\}, \{c,d\}, \{a\}, \{f\}} + \Phi_{H_n}^{\{a,f\}, \{c,d\}, \{b\}, \{e\}} - \Phi_{H_n}^{\{b,c\}, \{d,e\}, \{a\}, \{f\}} - \Phi_{H_n}^{\{a,c\}, \{d,f\}, \{b\}, \{e\}} \\
  d_n^{(2)} & = \Phi_{H_n}^{\{a,f\}, \{c,d\}, \{b\}, \{e\}} + \Phi_{H_n}^{\{a,f\}, \{b,e\}, \{c\}, \{d\}} - \Phi_{H_n}^{\{a,c\}, \{d,f\}, \{b\}, \{e\}} - \Phi_{H_n}^{\{a,b\}, \{e,f\}, \{c\}, \{d\}} \\
  d_n^{(3)} & = \Phi_{H_n}^{\{a,f\}, \{b,e\}, \{c\}, \{d\}} + \Phi_{H_n}^{\{b,e\}, \{c,d\}, \{a\}, \{f\}} - \Phi_{H_n}^{\{a,b\}, \{e,f\}, \{c\}, \{d\}} - \Phi_{H_n}^{\{b,c\}, \{d,e\}, \{a\}, \{f\}} \\
  e_n^{(1)} & = \Phi_{H_n}^{\{c\}, \{d\}} \\
  e_n^{(2)} & = \Phi_{H_n}^{\{b\}, \{e\}} \\
  e_n^{(3)} & = \Phi_{H_n}^{\{a\}, \{f\}} \\
  f_n^{(1)} & = \Phi_{H_n}^{\{b\}, \{c\}, \{d\}, \{e\}} \\
  f_n^{(2)} & = \Phi_{H_n}^{\{a\}, \{b\}, \{e\}, \{f\}} \\
  f_n^{(3)} & = \Phi_{H_n}^{\{a\}, \{c\}, \{d\}, \{f\}} \\
  g_n & = \Psi_{H_n} \\
  h_n^{(1)} & = \Phi_{H_n}^{\{b,c\}, \{d,e\}, \{a\}, \{f\}} + \Phi_{H_n}^{\{a,c\}, \{d,f\}, \{b\}, \{e\}} - \Phi_{H_n}^{\{b,d\}, \{c,e\}, \{a\}, \{f\}} - \Phi_{H_n}^{\{a,d\}, \{c,f\}, \{b\}, \{e\}} \\
  h_n^{(2)} & = \Phi_{H_n}^{\{a,c\}, \{d,f\}, \{b\}, \{e\}} + \Phi_{H_n}^{\{a,b\}, \{e,f\}, \{c\}, \{d\}} - \Phi_{H_n}^{\{a,d\}, \{c,f\}, \{b\}, \{e\}} - \Phi_{H_n}^{\{a,e\}, \{b,f\}, \{c\}, \{d\}} \\
  h_n^{(3)} & = \Phi_{H_n}^{\{a,b\}, \{e,f\}, \{c\}, \{d\}} + \Phi_{H_n}^{\{b,c\}, \{d,e\}, \{a\}, \{f\}} - \Phi_{H_n}^{\{a,e\}, \{b,f\}, \{c\}, \{d\}} - \Phi_{H_n}^{\{b,d\}, \{c,e\}, \{a\}, \{f\}} \\
\end{align*}

\begin{table}
\caption{Reductions of spanning forest polynomials for the computation of $c_2^{(2)}(\widetilde{C_n}(2,3))$}
\label{big table}
\centering
\begin{tabular}{|c|c|c|c|c|}
\hline
 & $\alpha$ & $\beta$ & $\gamma$ &  $\delta$ \\
\hline 
$a_n^{(1)}$ & $a_{n-2}^{(2)}$ & $a_{n-2}^{(5)}$ & * & $g_{n-2}$ \\
  $a_n^{(2)}$ & $a_{n-2}^{(1)}$ & $a_{n-2}^{(6)}$ & * & 0 \\
  $a_n^{(3)}$ & $a_{n-2}^{(1)}$ & $a_{n-2}^{(1)}$ & * & 0 \\
  $a_n^{(4)}$ & $a_{n-2}^{(2)}$ & $a_{n-2}^{(2)}$ & * & 0 \\
  $a_n^{(5)}$ & 0 & $a_{n-2}^{(3)}$ & 0 & 0 \\
  $a_n^{(6)}$ & 0 & $a_{n-2}^{(4)}$ & 0 & $g_{n-2}$ \\
  $b_n^{(1)}$ & $b_{n-2}^{(1)}$ & $b_{n-2}^{(2)}$ & $d_{n-2}^{(1)}$ & 0 \\
  $b_n^{(2)}$ & $-e_{n-2}^{(2)}$ & $b_{n-2}^{(3)}$ & $f_{n-2}^{(2)}$ & 0 \\
  $b_n^{(3)}$ & $b_{n-2}^{(1)}$ & $b_{n-2}^{(1)}$ & $d_{n-2}^{(1)}$ & 0 \\
  $d_n^{(1)}$ & $-f_{n-2}^{(1)}$ & $d_{n-2}^{(2)}$ & * & 0 \\
  $d_n^{(2)}$ & $-f_{n-2}^{(1)}$ & $d_{n-2}^{(3)}$ & * & $b_{n-2}^{(1)}$ \\
  $d_n^{(3)}$ & 0 & $d_{n-2}^{(1)} $ &  0 & $b_{n-2}^{(1)}$ \\
  $e_n^{(1)}$ & $e_{n-2}^{(2)}$ & $e_{n-2}^{(3)}$ & $f_{n-2}^{(2)}$ & 0 \\
  $e_n^{(2)}$ & $e_{n-2}^{(1)}$ & $e_{n-2}^{(1)}$ & $h_{n-2}^{(1)}$ & 0 \\
  $e_n^{(3)}$ & $e_{n-2}^{(2)}$ & $e_{n-2}^{(2)}$ & $f_{n-2}^{(2)}$ & 0 \\
  $f_n^{(1)}$ & $f_{n-2}^{(1)}$ & $f_{n-2}^{(3)}$ & * & $e_{n-2}^{(1)}$ \\
  $f_n^{(2)}$ & $f_{n-2}^{(1)}$ & $f_{n-2}^{(1)}$ & * & 0 \\
  $f_n^{(3)}$ & 0 & $f_{n-2}^{(2)}$ & 0 & $e_{n-2}^{(2)}$ \\
  $g_n$ & $g_{n-2}$ & $g_{n-2}$ & $a_{n-2}^{(3)} + a_{n-2}^{(4)}$ & 0 \\
  $h_n^{(1)}$ & $f_{n-2}^{(1)}$ & $h_{n-2}^{(2)}$ & * & 0 \\
  $h_n^{(2)}$ & $f_{n-2}^{(1)}$ & $h_{n-2}^{(3)}$ & * & $b_{n-2}^{(1)}$ \\
  $h_n^{(3)}$ & 0 & $h_{n-2}^{(1)}$ & * & 0 \\
\hline
\end{tabular}
\end{table}
The reader is strongly encouraged to draw each of these out with colours marking the parts.
For each of these polynomials there are at most four ways to assign the four edges out of $c$ and $d$.  Label the four ways as follows
\begin{itemize}
  \item[$\alpha$:] assign the outermost edges out of $c$ and $d$ to this polynomial but not the inner edges.
  \item[$\beta$:] assign the innermost edges of $c$ and $d$ (the ones connecting to $a$ and $f$) to this polynomial but not the outer edges.
  \item[$\gamma$:] assign all edges out of $c$ and $d$ to this polynomial.
  \item[$\delta$:] assign no edges out of $c$ and $d$ to this polynomial.
\end{itemize}
Then calculating by contracting and deleting the edges appropriately we can populate Table~\ref{big table}.
The entries in the table marked * are nonzero in general but won't be needed for the final calculation.  Now note that $\Psi^{1,3}_2\Psi^{12,32} = a_n^{(1)}b_n^{(1)}$.  Also we are working over $\mathbb{F}_2$, so the signs are irrelevant.  So, using the table and Lemma~\ref{coeff count lemma}, we can build a recurrence for $c_2^{(2)}(\widetilde{C_n}(2,3)$ as follows.  Let
\begin{align*}
  A_n = c_2^{(2)}(\widetilde{C_n}(2,3)) & = [a_n^{(1)}b_n^{(1)}]_2 \\
  B_n & = [a_n^{(2)}b_n^{(2)}]_2 \\
  C_n & = [a_n^{(3)}b_n^{(1)}]_2 \\
  D_n & = [g_nd^{(1)}_n]_2 \\
  E_n & = [a_n^{(1)}b_n^{(3)}]_2 \\
  F_n & = [a_n^{(4)}e_n^{(1)}]_2 \\
  G_n & = [a_n^{(1)}b_n^{(2)}]_2 \\
  H_n & = [g_nf_n^{(1)}]_2 \\
  I_n & = [g_nd_n^{(2)}]_2 \\
  J_n & = [a_n^{(2)}b_n^{(1)}]_2 \\
  K_n & = [a_n^{(2)}e_n^{(2)}]_2\\
  L_n & = [a_n^{(2)}e_n^{(3)}]_2\\
  M_n & = [a_n^{(2)}b_n^{(3)}]_2 \\
  N_n & = [a_n^{(3)}e_n^{(1)}]_2 \\
  P_n & = [g_nd_n^{(3)}]_2 \\
  Q_n & = [a_n^{(4)}b_n^{(1)}]_2 \\
  R_n & = [a_n^{(1)}e_n^{(1)}]_2\\
  S_n & = [a_n^{(1)}e_n^{(2)}]_2 \\
  T_n & = [a_n^{(1)}e_n^{(3)}]_2 \\
  U_n & = [a_n^{(2)}e_n^{(1)}]_2 \\
  V_n & = [g_nh_n^{(1)}]_2 \\
  W_n & = [g_nh_n^{(2)}]_2 
\end{align*}
Then
\begin{align*}
  A_n & = B_{n-2} + C_{n-4} + D_{n-2} \\
  B_n & = E_{n-2} + F_{n-4} + V_{n-4}\\
  C_n & = A_{n-2} + G_{n-2} \\
  D_n & = H_{n-2} + I_{n-2} \\
  E_n & = J_{n-2} + C_{n-4} + D_{n-2} \\
  F_n & = K_{n-2} + L_{n-2} \\
  G_n & = M_{n-2} + N_{n-4} \\
  H_n & = H_{n-2} + F_{n-2} + N_{n-2} \\
  I_n & = H_{n-2} + P_{n-2} + C_{n-2} + Q_{n-2} \\
  J_n & = G_{n-2} + Q_{n-4} \\
  K_n & = R_{n-2} \\
  L_n & = S_{n-2} + F_{n-4} + V_{n-4} \\
  M_n & = A_{n-2} + Q_{n-4} \\
  N_n & = S_{n-2} + T_{n-2} \\
  P_n & = D_{n-2} + C_{n-2} + Q_{n-2} \\
  Q_n & = B_{n-2} + J_{n-2} \\
  R_n & = L_{n-2} + N_{n-4} \\
  S_n & = U_{n-2} + V_{n-2} \\
  T_n & = K_{n-2} + N_{n-4} \\
  U_n & = T_{n-2} + F_{n-4} + V_{n-4}  \\
  V_n & = W_{n-2} + H_{n-2} \\
  W_n & = V_{n-4} + H_{n-2} + C_{n-2} + Q_{n-2}
\end{align*}
In the above computation we used that $0 = f_n^{(2)}g_n = a_n^{(4)}e_n^{(2)} = a_n^{(3)}e_n^{(2)}$ and where convenient did a second reduction on terms involving $a_n^{(5)}$, $a_n^{(6)}$ and $h_n^{(3)}$ to avoid having even more equations.  This recurrence could now be further simplified or solved and the required many initial terms could be explicitly computed.  This $c_2^{(2)}(\widetilde{C_n}(2,3))$ recurrence illustrates how calculations using this method explode in size.

\section{$\widetilde{C_n}(1,k)$ for $k \geq 4$ and other $\widetilde{C_n}(i,j)$}

What about $\widetilde{C_n}(1,k)$ for $k \geq 4$?  For $k=4,5,6$ we again get a finiteness result.  For $k\geq 7$ we can not deal with enough initial edges to put the graph in an appropriate form to build recurrences and so the method fails completely.  Similarly, we also get a finiteness result for $\widetilde{C_n}(2,4)$, $\widetilde{C_n}(2,5)$, and $\widetilde{C_n}(3,4)$, while the method fails for $\widetilde{C_n}(i,j)$ with larger gaps.

\begin{figure}
\includegraphics{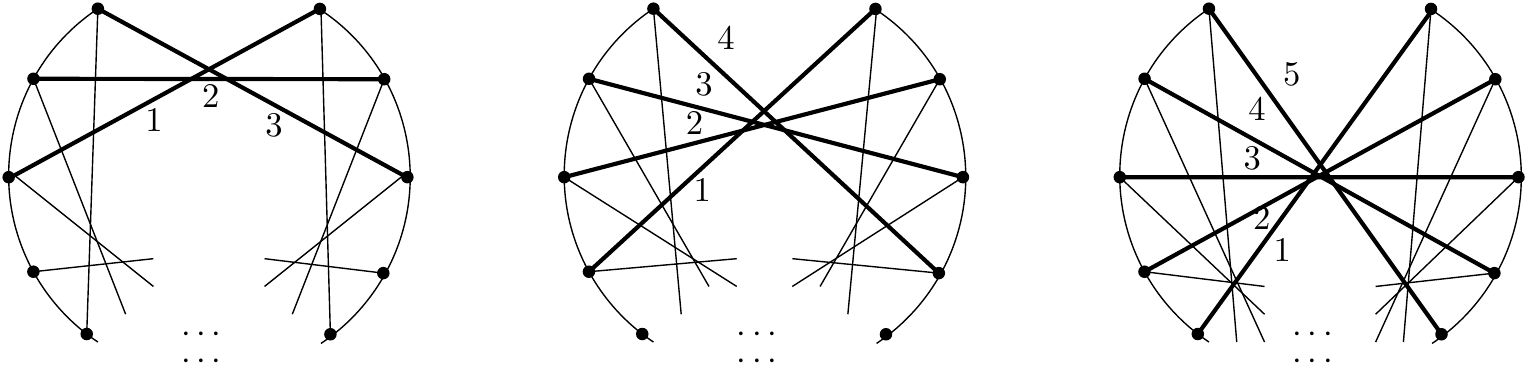}
\caption{$\widetilde{C_n}(1,k)$ for $4 \leq k \leq 6$}\label{k from 4 to 6}
\end{figure}

\begin{prop}\label{finite no 2}
  Fix any prime $p$ and $4 \leq k \leq 6$.  It is a finite calculation to determine
   $c_2^{(p)}(\widetilde{C_n}(1,k))$ for all $n$.
\end{prop}

\begin{proof}
  Figure~\ref{k from 4 to 6} illustrates $\widetilde{C_n}(1,k)$ for $4\leq k \leq 6$.  For $k=4$ begin with $\Psi^{1,3}_2\Psi^{12,32}$.  For $k=5$ begin with $\Psi^{12,34}\Psi^{13,24}$ and for $k=6$ begin with ${}^5\Psi(1,2,3,4,5)$.  In all three cases, the result is a sum of products of spanning forest polynomials where the underlying graph has the upper two vertices 2-valent, the next $k-1$ vertices on each side are 3-valent, and the rest are 4-valent.  In all three cases removing the two upper vertices gives the same graph but for $n-2$.  In all three cases the spanning forest polynomials have at most the top $k$ vertices on each side in their partitions.  Finally, the edges from the top two vertices go to the second to the top vertex and to the $k+1$st from the top vertex.  Therefore, the procedure described in Proposition~\ref{finite no 1} works with only the change that $\mathcal{P}$ is given by all set partitions of subsets of the top $k$ vertices on each side.
\end{proof}

The difficulty with $k\geq 7$ is that we need to start off by dealing with enough edges that there are no remaining edges connecting the top vertices of one side to the other side.  Then the graph has a linear structure which is suitable to recursion.  To get the graph into this form for $k\geq 7$ we would need to be able to begin with at least $6$ edges dealt with which is not possible in general.

As with Proposition~\ref{finite no 1}, without some hidden structure making life easy, these calculations quickly become infeasible either by hand or by computer.  There is no obviously suggestive pattern in the first few terms of $c_2^{(2)}(\widetilde{C_n}(1,4))$ or $c_2^{(2)}(\widetilde{C_n}(1,5))$.

\begin{figure}
\includegraphics{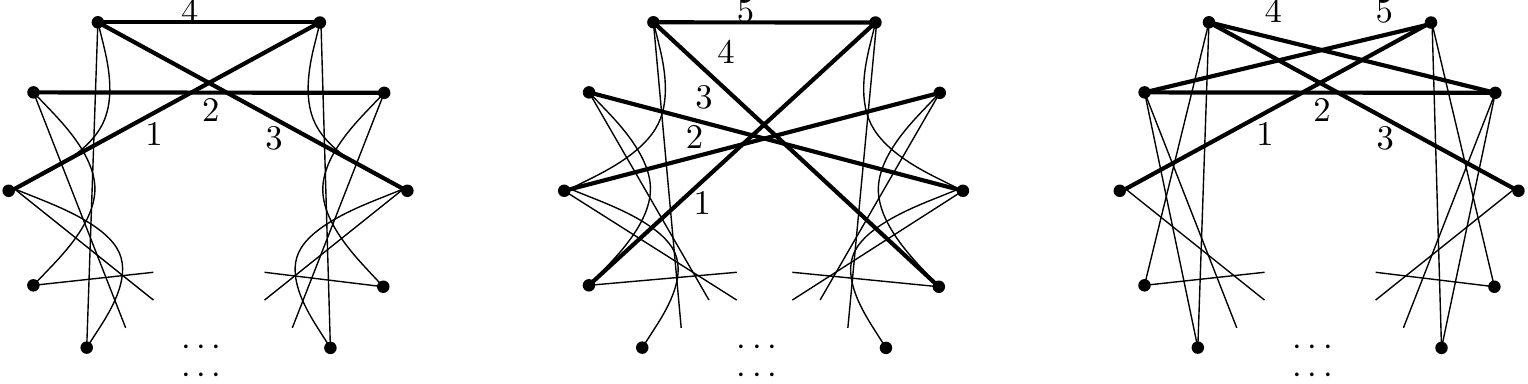}
\caption{$\widetilde{C_n}(j,k)$ for $(j,k)\in \{(2,4), (2,5), (3,4)\}$}\label{j k that work}
\end{figure}

\begin{prop}\label{finite no 3}
  Fix any prime $p$ and $(j,k) \in \{(2,4), (2,5), (3,4)\}$.  It is a finite calculation to determine
   $c_2^{(p)}(\widetilde{C_n}(j,k))$ for all $n$.
\end{prop}

\begin{proof}
  Figure~\ref{j k that work} illustrates the graphs in question.  For $(j,k)=(2,4)$ begin with $\Psi^{12,34}\Psi^{13,24}$.  For the other two begin with ${}^5\Psi(1,2,3,4,5)$.  Argue as before, using the top $4$ vertices in the first and last case and the top 5 vertices in the middle case.
\end{proof}

Again, the difficulty with other values of $(j,k)$ is that to get the graph into the correct form we would need to be able to begin by dealing with more edges than is possible.

\section{$\widetilde{C_{2k+2}}(1,k)$}

Another place to look for tractible calculations is by fixing the relationship between $k$ and the number of vertices in $C_n(1,k)$.
Note that $C_{n}(1,k) = C_n(1,n-k)$ so let us restrict our attention to $n\geq 2k$.  $C_{2k}(1,k)$ is degenerate in the sense that it is not 4-regular (or it had double edges) since $k|2k$.  $C_{2k+1}(1,k)$ is isomorphic to $C_{2k+1}(1,2)$ where the circle edges of the former become the chord edges of the latter and vice versa.  Therefore the first interesting case to look at is $\widetilde{C_{2k+2}}(1,k)$.

\begin{prop}
  \mbox{}
  \begin{enumerate}
    \item $c_2^{(2)}(\widetilde{C_{2k+2}}(1,k)) = 0$ for $k \geq 3$.
    \item For any prime $p$, it is a finite calculation to determine $c_2^{(p)}(\widetilde{C_{2k+2}}(1,k))$ for all $k$.
  \end{enumerate}
\end{prop}

\begin{figure}
\includegraphics{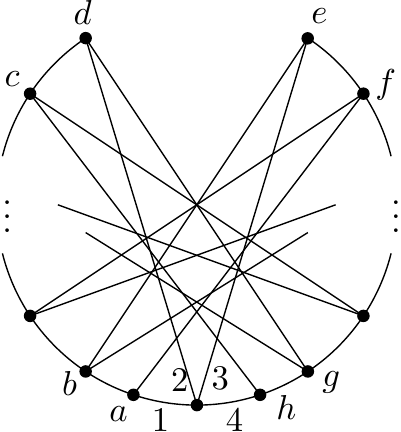}
\caption{$\widetilde{C_{2k+2}}(1,k)$}\label{C2n+2}
\end{figure}

\begin{proof}
Label $\widetilde{C_{2k+2}}(1,k)$ as in Figure~\ref{C2n+2}.  Calculate
\[
\Psi^{12,34} = \pm\left(\Phi^{\{d,e\},\{a,h\}} - \Phi^{\{d,h\}, \{a,e\}}\right)
\]
and
\[
\Psi^{13,24} = \pm\left(\Phi^{\{d,e\},\{a,h\}} - \Phi^{\{a,d\}, \{e,h\}}\right).
\]
Again we proceed by Lemma \ref{coeff count lemma}.  Let $L$ be $\widetilde{C_{2k+2}}(1,k)$ with edges $1$, $2$, $3$ and $4$ removed, see figure~\ref{L}.  Write $L_k$ when it is useful to make the size dependence explicit.

\begin{figure}
\includegraphics{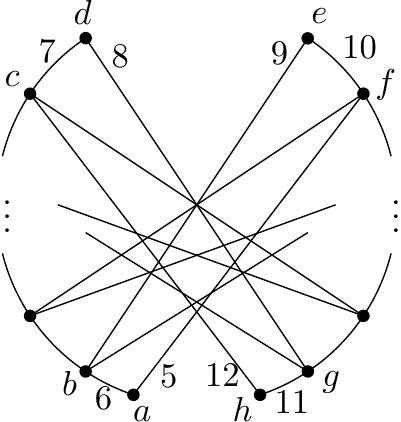}
\caption{$L$}\label{L}
\end{figure}

For $p=2$ we may restrict ourselves to spanning forests symmetric about a central reflection.  To keep the parts connected in both $\Psi^{12,34}$ and $\Psi^{13,24}$ we must use exactly one of edges $5$ and $6$ in each and similarly for edges $7$ and $8$, for edges $9$ and $10$ and for edges $11$ and $12$.  By left-right symmetry there are four possibilities: 
\begin{itemize}
  \item \textbf{Edges $6$, $7$, $10$, and $11$ are assigned to $\Psi^{12,34}$ or edges $5$, $8$, $9$, and $12$ are assigned to $\Psi^{12,34}$.}  In both these cases, performing the contractions and deletions, both $\Psi^{12,34}$ and $\Psi^{13,24}$ become the same spanning forest polynomial expressions they began as, but on $L_{k-2}$ rather than on $L_k$.  Therefore these two cases have an even contribution.
  \item \textbf{Edges $6$, $8$, $9$, and $11$ are assigned to $\Psi^{12,34}$ or edges $5$, $7$, $10$, and $12$ are assigned to $\Psi^{12,34}$.} 
    In both these cases $\Phi^{\{d,h\}, \{a,e\}}$ and $\Phi^{\{a,d\}, \{e,h\}}$ contribute nothing because they give contractions which identify distinct parts.  What remains is the contribution from $\Phi_{L_k}^{\{d,e\},\{a,h\}}$ for each of $\Psi^{12,34}$ and $\Psi^{13,24}$.  That is, we only need to consider assigning these edges in $(\Phi_{L_k}^{\{d,e\},\{a,h\}})^2$.  This is $\Phi_{L_{k-2}}^{\{a\},\{h\}}\Phi_{L_{k-2}}^{\{d\},\{e\}}$ in both cases, where $L_{k-2}$ is labelled as in Figure~\ref{L}, not as a subgraph of $L_k$.  Therefore these two cases contribute the same thing and so their sum is even.
\end{itemize}
Taking all the cases together
 $c_2^{(2)}(\widetilde{C_{2k+2}}(1,k)) = 0$ for $k\geq 3$.

Returning to general $p$ we have a sum of products of spanning forest polynomials of $L_k$ with parts using only vertices $a$, $d$, $e$ and $h$. All possible deletions and contractions of edges $5$ through $12$ give spanning forests on $L_{k-2}$ involving only the top two and bottom two vertices.  Thus as in Propositions~\ref{finite no 1}, \ref{finite no 2}, and \ref{finite no 3} for any fixed $p$ it is a finite calculation to determine $c_2^{(p)}(\widetilde{C_{2k+2}}(1,k))$ for all $k$.
\end{proof}

\section{Conclusion}

Circulant graphs are an interesting playground for working on the $c_2$ invariant because they give a section through different difficulties including both the simplest nontrivial case -- the zigzags, and what appear to be the most difficult.  None-the-less they are highly symmetric.  

For circulant graphs with sufficiently small gaps, initial edge reductions can be done so as to put the remainder of the decompletion of the graph in a linear form which is suitable for recursion.  This is already interesting since it gives an in principle way to determine the $c_2$ invariants for these families of graphs.  Only a few cases are amenable to hand calculation.  A computer implementation to extend the practical applicability of these ideas remains to be done.

Note that these recurrences are all for fixed $p$.  That is, for fixed $p$ there are systems of linear recurrences determining $c_2^{(p)}$ for certain families of decompleted circulant graphs.  This restricts the form of $c_2^{(p)}$, as a function of the parameter $n$ determining the family, to be something which can come from the solution to a system of linear recurrences, essentially something which comes from taking $n$th powers of appropriate matrices.
If, on the other hand, a graph is fixed and $p$ is varied we know that the relationship must be much more complicated as the $c_2$ invariant can give coefficients of modular forms \cite{BrS, BrS3}.  This depth must come from the increasing complexity of the different recurrences as $p$ changes.

\bibliographystyle{plain}
\bibliography{main}

\end{document}